\documentclass[11pt]{amsart}

\usepackage{epigamath}

\usepackage[english]{babel}

\numberwithin{equation}{section}

\usepackage[shortlabels]{enumitem}
\setlist[enumerate,1]{label={\rm(\arabic*)}, ref={\rm\arabic*}} 

\usepackage{amsfonts, amssymb, amscd}
\usepackage{amsmath}
\usepackage{verbatim}
\usepackage{amssymb}
\usepackage{mathrsfs}
\usepackage{graphicx}
\usepackage{cite}
\usepackage[all]{xy}

\usepackage{graphicx}
\usepackage{float}

\newtheorem{theorem}{Theorem}[section]
\newtheorem{lemma}[theorem]{Lemma}
\newtheorem{proposition}[theorem]{Proposition}
\newtheorem{corollary}[theorem]{Corollary}

\theoremstyle{definition}
\newtheorem{definition}[theorem]{Definition}

\theoremstyle{remark}

\newtheorem{remark}[theorem]{Remark}

\def\C {{\mathbb C}}
\def\CC {{\mathbb C}}
\def\R {{\mathbb R}}
\def\RR {{\mathbb R}}
\def\Z {{\mathbb Z}}
\def\ZZ {{\mathbb Z}}
\def\PP {{\mathbb P}}
\def\Q {{\mathbb Q}}
\def\KK {{\mathbb K}}
\def\ii {{\rm i}}

\def\res {{\operatorname{Res}}}

\allowdisplaybreaks[3]

\DeclareMathOperator{\Spec}{Spec}
\DeclareMathOperator{\FM}{FM}
\DeclareMathOperator{\Sol}{Sol}
\DeclareMathOperator{\MB}{MB}
\renewcommand{\Box}{\operatorname{Box}}
\DeclareMathOperator{\Star}{Star}
\DeclareMathOperator{\Ann}{Ann}
\DeclareMathOperator{\Vol}{Vol}
\DeclareMathOperator{\supp}{supp}
\DeclareMathOperator{\Hom}{Hom}
\DeclareMathOperator{\bbGKZ}{bbGKZ}

\def\es{\mathrm{es}}
\def\orb{\mathrm{orb}}
\def\b{\mathrm{b}}
\def\ch{\mathrm{ch}}
\def\lowsim{\vbox to 0pt{\vss\hbox{$\scriptstyle\sim$}\vskip-1.8pt}}

\newcommand{\supth}[1]{\ensuremath{#1^{\mathrm{th}}}}

\EpigaVolumeYear{9}{2025} \EpigaArticleNr{11} \ReceivedOn{August 3, 2023}
\InFinalFormOn{July 16, 2024}
\AcceptedOn{October 31, 2024}

\title{Analytic continuation of better-behaved GKZ systems and Fourier--Mukai transforms}
\titlemark{Analytic continuation and Fourier--Mukai}

\author{Zengrui Han}
\address{Department of Mathematics, Rutgers University, Piscataway, NJ 08854, USA}
\email{zh223@math.rutgers.edu}

\authormark{Z.~Han}

\AbstractInEnglish{We study the relationship between solutions to better-behaved GKZ hypergeometric systems near different large radius limit points, and their geometric counterparts given by the $K$-groups of the associated toric Deligne--Mumford stacks. We prove that the $K$-theoretic Fourier--Mukai transforms associated to toric wall-crossing coincide with analytic continuation transformations of Gamma series solutions to the better-behaved GKZ systems, which settles a conjecture of Borisov and Horja.}


\begin{document}

%%%%%%%%%%%%%%%%%%%%%%%%%%%%%%%
% Title page
%%%%%%%%%%%%%%%%%%%%%%%%%%%%%%%

%\removeabove{}
%\removebetween{}
%\removebelow{}

\maketitle

\begin{prelims}

\DisplayAbstractInEnglish

\bigskip

\DisplayKeyWords

\medskip

\DisplayMSCclass

\end{prelims}

%%%%%%%%%%%%%%%%%%%%%
% Table of Contents
%%%%%%%%%%%%%%%%%%%%%

\newpage

\setcounter{tocdepth}{1}

\tableofcontents

%%%%%%%%%%%%%%%%%%%%%
% Content begins here
%%%%%%%%%%%%%%%%%%%%%

\section{Introduction}\label{sec.intro}

Borisov and Horja \cite{BH} introduced a better-behaved version of the hypergeometric systems of Gel'fand, Kapranov and Zelevinsky \cite{GKZoriginal}, in the sense that the solution spaces always have the expected dimensions in contrast to the original version where a rank-jumping phenomenon may occur. The better-behaved GKZ systems (bbGKZ systems) are thus more suitable for any kind of functorial consideration. Furthermore, it turns out that they are closely related to the moduli theory of hypersurfaces in toric varieties and play a crucial role in toric mirror symmetry. For example, they describe the Gauss--Manin
% query dash
systems associated to the Landau--Ginzburg mirror potentials of toric Deligne--Mumford stacks (see \textit{e.g.}
% query e.g.
\cite[Section~5.1]{CCIT}).

Due to the non-compactness of the toric Deligne--Mumford stacks we consider, there exist two such systems $\bbGKZ(C,0)$ and $\bbGKZ(C^{\circ},0)$, where the latter should be considered as a compactly supported version of the former. In \cite{BHconj}, Borisov and Horja formulated a pair of conjectures: one regarding the duality between these two systems (which was settled in full generality in \cite{BHan}) and another regarding the connection between the analytic continuation of solutions to these systems and certain Fourier--Mukai transforms. This paper is focused on the latter conjecture.

To formulate the main result more precisely, we first introduce our combinatorial setting and review the definition of the bbGKZ systems.

Let $C$ be a finite rational polyhedral cone in a vector space $N_{\RR}:=N\otimes_{\ZZ}\RR$, where $N$ is a lattice. We assume the ray generators of $C$ are lattice points in $N$ and lie on a primitive hyperplane $\deg(-)=1$, where $\deg$ is a linear function on $N$. These data define an affine Gorenstein toric variety $X=\Spec(\C[C^{\vee}\cap N^{\vee}])$. Consider a set $\{v_i\}_{i=1}^n$ of lattice points of degree $1$ which includes all ray generators of the cone $C$; a simplicial subdivision $\Sigma$ of $C$ based on this set gives a crepant resolution $\PP_{\Sigma}\rightarrow X$. Generally speaking, $\PP_{\Sigma}$ is a smooth toric Deligne--Mumford stack rather than a smooth toric variety.

\begin{definition}\label{defGKZ}
	Consider the system of partial differential equations on the collection of functions $\{\Phi_c(x_1,\ldots,x_n)\}$ in complex variables $x_1,\ldots, x_n$, indexed by the lattice points in $C$:
	\begin{align*}
		\partial_i\Phi_c=\Phi_{c+v_i},\quad \sum_{i=1}^n\langle\mu,v_i\rangle x_i\partial_i\Phi_c+\langle\mu,c\rangle\Phi_c=0
	\end{align*}
	for all $\mu\in N^{\vee}$, $c\in C$ and $i=1,\ldots,n$. We denote this system by $\bbGKZ(C,0)$. Similarly, by considering lattice points in the interior $C^{\circ}$ only, we can define $\bbGKZ(C^{\circ},0)$.
\end{definition}

Borisov and Horja introduced Gamma series solutions with values in the complexified $K$-group (or, equivalently, the orbifold cohomology) of the toric stack $\PP_{\Sigma}$ to these systems which is locally defined in a neighborhood of the large radius limit point corresponding to $\Sigma$  in \cite{BHconj}. As a result, the $K$-groups provide a local integral structure of the local systems of solutions to the bbGKZ systems. 

In this paper, we show that this local integral structure is compatible with the natural analytic continuation of solutions to the bbGKZ systems. More precisely, we prove that under the isomorphisms provided by Gamma series (the \textit{mirror symmetry maps}), the analytic continuation of solutions to bbGKZ systems from the neighborhood of one triangulation $\Sigma_+$ to the neighborhood of another adjacent triangulation $\Sigma_-$ coincides with the $K$-theoretic Fourier--Mukai transform associated to the flop $\PP_{\Sigma_-}\dashrightarrow\PP_{\Sigma_+}$.

\begin{theorem}[= Theorems~\ref{conj for usual system} and~\ref{conj for dual}]
        The following diagrams commute:
	\begin{align*}
		\xymatrix{
K_0\left(\PP_{\Sigma_+}\right)^{\vee}\ar[d]^{\FM^{\vee}}\ar[r]^-{-\circ\Gamma_+} & \Sol(\bbGKZ(C,U_+))\ar[d]^{\MB} \\
K_0\left(\PP_{\Sigma_-}\right)^{\vee}\ar[r]^-{-\circ\Gamma_-} & \Sol(\bbGKZ(C,U_-))\rlap{,} \\
    }
	\end{align*}
	\begin{align*}
    \xymatrix{
K_0^c\left(\PP_{\Sigma_+}\right)^{\vee}\ar[d]^{(\FM^c)^{\vee}}\ar[r]^-{-\circ\Gamma^{\circ}_+} & \Sol(\bbGKZ(C^{\circ}),U_+)\ar[d]^{\MB^c} \\
K_0^c\left(\PP_{\Sigma_-}\right)^{\vee}\ar[r]^-{-\circ\Gamma^{\circ}_-} & \Sol(\bbGKZ(C^{\circ}),U_-)\rlap{,} \\
    }
\end{align*}
        where the horizontal arrows are mirror symmetry maps and  $\FM$  and $\MB$  $($respectively, $\FM^c$ and $\MB^c)$ denote the Fourier--Mukai transforms $($respectively, analytic continuation transformations of solutions$)$.
        % query correct? 
\end{theorem}

This phenomenon was first observed in the Ph.D.\ thesis of Horja \cite{horjathesis} and was later studied by Borisov and Horja in \cite{MellinBarnes}. However, the authors were using the original version of the GKZ systems, and the map between the dual of the $K$-theory and the solution space is not necessarily an isomorphism due to the rank-jumping phenomenon at non-generic parameters (see \textit{e.g.}~\cite{MMW}). The advantage of the bbGKZ systems is that  the mirror symmetry maps from the duals of the $K$-groups to the solution spaces are always isomorphisms.

The motivation behind these works comes from Kontsevich's homological mirror symmetry, see \cite{kontsevich}, which predicts that the fundamental group of the moduli space of the complex structures on the one side naturally acts on the bounded derived category of coherent sheaves on the other side. This suggests the existence of an isotrivial family of triangulated categories over the complex moduli space. For the toric case, on the level of the Grothendieck groups, this family gives us the local system of solutions to the bbGKZ systems. The main results of \cite{MellinBarnes} and this paper hence fit within this framework. Nevertheless, a general construction of such a family on the level of triangulated categories is currently unknown, although progress has been made in the quasi-symmetric case by \v{S}penko and Van den Bergh in \cite{spenko-van den bergh} (see also \cite{spenko survey} for a survey on this).

The paper is organized as follows. In Section~\ref{sec.basics}, we review basic facts about bbGKZ systems, toric Deligne--Mumford stacks and toric wall-crossing, and we fix notation that will be used throughout this paper. In Section~\ref{sec.ac}, we compute the analytic continuation of Gamma series solutions to $\bbGKZ(C,0)$. In Section~\ref{sec.fm}, we compute the Fourier--Mukai transform associated to the toric wall-crossing $\PP_{\Sigma_-}\dashrightarrow\PP_{\Sigma_+}$ and match it with the analytic continuation computed in Section~\ref{sec.ac}. Finally, in Section~\ref{sec.dual}, we make use of the duality result in \cite{BHan} to prove the analogous result for the dual system $\bbGKZ(C^{\circ},0)$.

\section*{Acknowledgements} The author would like to thank his advisor Lev Borisov for suggesting this problem, and for consistent support, helpful discussions and useful comments throughout the preparation of this paper. The author would also like to thank the referee, whose suggestions greatly improved this paper, for the careful and thoughtful reading of the text.

\section{Better-behaved GKZ systems and toric wall-crossing}\label{sec.basics}

In this section, we collect basic facts about better-behaved GKZ systems and toric wall-crossing. The main references are \cite{BCS,BH,MellinBarnes, CIJ, GKZbook}.

\subsection{Toric Deligne--Mumford stacks and twisted sectors}

First we recall the definition of smooth toric Deligne--Mumford stacks (or toric orbifolds) following Borisov--Chen--Smith \cite{BCS}. See also \cite[Section 3]{BHconj}.

\begin{definition}\label{definition of toric stack}
	Suppose $C$, $N$, $\{v_1,\ldots,v_n\}$ and $\Sigma$ are combinatorial data defined in Section~\ref{sec.intro}. Consider the open subset $U$ of $\CC^n$ defined by
	\begin{align*}
		U=\left\{(z_1,\ldots,z_n)\in\CC^n:\ \{i:z_i=0\}\in\Sigma\right\}
	\end{align*}
	and the subgroup $G$ of $(\CC^*)^n$ defined by
	\begin{align*}
		G=\left\{(\lambda_1,\ldots,\lambda_n):\ \prod_{i=1}^n \lambda_i^{\langle m, v_i \rangle}=1,\forall m\in N^{\vee}\right\}
	\end{align*}
	The smooth toric Deligne--Mumford stack $\PP_{\Sigma}$ associated to $C$, $N$, $\{v_1,\ldots,v_n\}$ and $\Sigma$ is defined to be the stack quotient of $U$ by $G$.
\end{definition}

Recall that the \textit{twisted sectors} of a Deligne--Mumford stack are defined to be the connected components of its inertia stack. In the case of smooth toric Deligne--Mumford stacks, we have the following combinatorial description of twisted sectors.

\begin{proposition}
	There is a $1$-$1$ correspondence between the set of twisted sectors of\, a smooth toric Deligne--Mumford stack $\PP_{\Sigma}$ with the set $\Box(\Sigma)$ of \textit{twisted sectors} of\, $\Sigma$ to the set of lattice points $\gamma\in N$ which can be written as $\gamma=\sum_{j=1}^n \gamma_j v_j$, where all coefficients satisfy $\gamma_j\in[0,1)$, such that $\{j:\ \gamma_j\not=0\}$ is a cone in $\Sigma$.
\end{proposition}
\begin{proof}
	See \cite[Proposition~4.7]{BCS}.
\end{proof}

From now on we will use the symbol $\gamma$ to denote either a connected component of the inertia stack or its corresponding lattice points in $C\cap N$. Now we give an alternative characterization of twisted sectors following \cite{Iritani}. For any lattice point $c\in C\cap N$, we define
\begin{align*}
	\KK_c:&=\left\{(l_i)\in\Q^n:\ \sum_{i=1}^n l_i v_i=-c,\  \{i:l_i\not\in\Z\}\text{ is a cone in }\Sigma\right\}\\
	&=\bigcup_{\gamma\in\Box(\Sigma)}L_{c,\gamma}, 
\end{align*}
where
\begin{align*}
	L_{c,\gamma}:=\left\{(l_i)\in\Q^n:\ \sum_{i=1}^n l_i v_i=-c,\  l_i\equiv \gamma_i\mod\Z\right\}
\end{align*}
Clearly, the set $L:=L_{0,0}$ acts on $\KK_c$ by translation. The following characterization of twisted sectors can be found in \cite[Section~3.1.3]{Iritani}.

\begin{lemma}\label{associated twisted sector}
	%There is a 1-1 correspondence between $\KK_c/L$ and $\Box(\Sigma)$.
	There is an injection $\KK_c/L\hookrightarrow\Box(\Sigma)$, and the image of this map consists of twisted sectors $\gamma$ such that the set $L_{c,\gamma}$ is non-empty.
\end{lemma}
\begin{proof}
	The map $\KK_c\rightarrow\Box(\Sigma)$ defined by $(l_i)\mapsto \sum_{i=1}^n \{l_i\}v_i$ clearly factors through $\KK_c/L$. Now take a twisted sector $\gamma\in\Box(\Sigma)$; then any element in the lattice $L_{c,\gamma}$ is mapped to $\gamma$ due to the condition $l_i\equiv \gamma_i\mod\Z$. 
\end{proof}

Thus each twisted sector $\gamma$ with $L_{c,\gamma}\not=\emptyset$ is represented by elements in the lattice $\KK_c$. We call such representatives the \textit{liftings} of $\gamma$.

\subsection{Orbifold cohomology and Grothendieck K-groups}

Next we recall the combinatorial description of the orbifold cohomology ring $H_{\orb}^*(\PP_{\Sigma})$ and the Grothendieck $K$-group $K_0(\PP_{\Sigma})$ of derived category of coherent sheaves  of the toric stack $\PP_{\Sigma}$. To give the definition of orbifold cohomology, we need the following Stanley--Reisner-type presentation of the usual cohomology of twisted sectors of $\PP_{\Sigma}$ that can be found in \cite[Proposition 2.3]{BHconj}.

\begin{proposition}
	The cohomology space $H_{\gamma}$ of the twisted sector $\gamma$ is naturally isomorphic to the the quotient of the polynomial ring $\C[D_i:i\in\Star(\sigma(\gamma))\backslash\sigma(\gamma)]$ by the relations
	\begin{align*}
		\prod_{j \in J} D_j, \quad J \notin \Star(\sigma(\gamma)) \quad \text {and}\quad \sum_{i \in \Star(\sigma(\gamma)) \backslash \sigma(\gamma)} \mu\left(v_i\right) D_i, \quad \mu \in \Ann\left(v_i, i \in \sigma(\gamma)\right), 
	\end{align*}
	where $D_i$ denotes the torus-invariant divisor corresponding to the ray generated by $v_i$, $\sigma(\gamma)$ denotes  $($the set of indices of\;$)$ the minimal cone in $\Sigma$ containing the twisted sectors $\gamma$, $\Star(\sigma(\gamma))$ denotes the set of all cones in $\Sigma$ containing $\sigma(\gamma)$ as a subcone $($\textit{i.e.}, the usual star construction$)$ and $\Ann(v_i,i\in\sigma(\gamma))$ denotes the set of linear functions in $N^{\vee}$ that take zero values on the $v_i$ for $i\in\sigma(\gamma)$.
	\end{proposition}

The orbifold cohomology $H_{\orb}^*(\PP_{\Sigma})$ of $\PP_{\Sigma}$ is then defined to be the direct sum $\bigoplus_{\gamma}H_{\gamma}$ of the cohomology spaces of the twisted sectors. Note that we ignore the degree shifting given by the age of each twisted sector, since for our purpose there will be no difference.

Next we give a combinatorial description of the Grothendieck $K$-group $K_0(\PP_{\Sigma})$ of $\PP_{\Sigma}$ following \cite[Theorem 4.10]{Ktheory}. Note that there is a natural multiplication operation defined by the alternative sum of higher Tor-sheaves. Recall that the toric stack $\PP_{\Sigma}$ is defined as the stack quotient $[U/G]$ as in Definition~\ref{definition of toric stack}. It is well known that the category of coherent sheaves on $[U/G]$ is equivalent to the category of $G$-linearized coherent sheaves on $U$. For each $1$-dimensional cone $\RR_{\geq 0}v_i$, $i=1,\ldots,n$, we define a corresponding $G$-linearized invertible sheaf $\mathcal{L}_i$ whose underlying sheaf on $U$ is simply $\mathcal{O}_U$, and the isomorphism $\mathcal{O}_U\rightarrow g^*\mathcal{O}_U$ for $g=(\lambda_1,\ldots,\lambda_n)\in G$ is given by mapping $1$ to $\lambda_i$. Note that in the special case where $\PP_{\Sigma}$ is a smooth toric variety, $\mathcal{L}_i$ is exactly the same as the line bundle $\mathcal{O}(D_i)$ corresponding to the torus-invariant divisor $D_i$ corresponding to the $1$-dimensional cone $\RR_{\geq 0}v_i$.

\begin{proposition}
	The Grothendieck $K$-group $K_0(\PP_{\Sigma})$ with complex coefficients is generated by the classes $R_i:=[\mathcal{L}_i]$ for $i=1,\ldots,n$. More precisely, $K_0(\PP_{\Sigma})$ is isomorphic to the quotient of the Laurent polynomial ring $\C[R_i^{\pm 1}]$ by the ideal generated by 
	\begin{align*}
		\prod_{i=1}^n R_i^{\mu(v_i)}-1,\quad \mu\in N^{\vee}\quad\text{and}\quad \prod_{i\in I}(1-R_i),\quad I\not\in\Sigma. 
	\end{align*}
\end{proposition}

The orbifold cohomology ring and the $K$-group of $\PP_{\Sigma}$ are related by the Chern character $\ch\colon K_0(\PP_{\Sigma})\xrightarrow{\lowsim}H_{\orb}^*(\PP_{\Sigma})$. A combinatorial description can be found in \cite[Proposition 3.6]{BHconj}.

\subsection{Secondary fans and toric wall-crossing}

Before we describe the toric wall-crossing setting in this paper, we briefly recall the basic definitions and properties of secondary fans.

Let $\Sigma$ be a triangulation of the cone $C$ based on the set of vertices $\{v_1,v_2,\ldots,v_n\}$. We define the characteristic function $\varphi_{\Sigma}\colon\{v_1,v_2,\ldots,v_n\}\rightarrow\R$
% query ldost for lists
by $\varphi_{\Sigma}(v_i):=\sum \Vol(\sigma)$, where $\Vol(\sigma)$ denotes the volume of $\sigma$ and the sum is taken over all simplexes $\sigma$ in $\Sigma$ that contain $v_i$ as a vertex. Note that $\varphi_{\Sigma}$ could be seen as a lattice point in $\R^n$.

\begin{definition}
	The \textit{secondary polytope} of the cone $C$ is defined to be the convex hull of $\varphi_{\Sigma}$ for all triangulations $\Sigma$ in $\R^n$, and the \textit{secondary fan} of $C$ is defined to be the normal fan of the secondary polytope.
\end{definition}

The following basic properties of secondary polytopes and fans could be found in \cite[Chapter 7]{GKZbook}.

\begin{proposition}
	The vertices of the secondary polytope of\, $C$ $($equivalently, the maximal cones of the secondary fan of\, $C)$ are in $1$-$1$ correspondence with regular triangulations of\, $C$.
\end{proposition}

\begin{remark}
	It is clear that the intersection of two maximal cones in the secondary fan is either $\{0\}$ or a cone of codimension $1$. In the latter case, the corresponding triangulations are said to be \textit{adjacent} to each other.
\end{remark}

Let $\Sigma_-$ and $\Sigma_+$ be two adjacent triangulations of the cone $C$ in the sense that the intersection of their corresponding maximal cones (which we denote by $C_{\Sigma-}$ and $C_{\Sigma_+}$, respectively) in the secondary fan is a codimension $1$ cone. Then there exists a circuit (\textit{i.e.}, a minimal linearly dependent set) $I$ defined by an integral linear relation
\begin{align*}
	h_1 v_1 +\cdots + h_n v_n =0
\end{align*}
with $I=I_+\sqcup I_-$, where $I_+=\{i:h_i>0\}$ and $I_-=\{i:h_i<0\}$. Moreover, the linear relation $h=(h_1,\ldots,h_n)$ gives the defining equation of the codimension $1$ subspace that is spanned by the intersection of the maximal cones corresponding to $\Sigma_{\pm}$ . We specify a special class of cones in the fan $\Sigma_{\pm}$. 

\begin{definition}
        A maximal cone in $\Sigma_{\pm}$ of the form $\mathcal{F}\sqcup(I\backslash i)$, where $i\in I_{\pm}$ and $\mathcal{F}\subseteq\{1,2,\ldots,n\}\backslash I$, is called an \textit{essential maximal cone} in $\Sigma_{\pm}$, and the set $\mathcal{F}$ is called the \textit{separating set} of the essential maximal cone. We denote the set of essential cones in $\Sigma_{\pm}$ by $\Sigma_{\pm}^{\es}$. If the minimal cone $\sigma(\gamma_{\pm})$ of a twisted sector $\gamma_{\pm}$ is a subcone of an essential maximal cone in $\Sigma$, then we say $\gamma_{\pm}$ is an \textit{essential twisted sector}; we denote the set of essential twisted sectors by $\Box(\Sigma_{\pm}^{\es})$.
\end{definition}

\begin{definition}
        Let $\sigma_{\pm}$ be essential maximal cones in $\Sigma_{\pm}$. We say that $\sigma_+$ and $\sigma_-$ are \textit{adjacent} if they have the same separating set $\mathcal{F}$. Equivalently, $\sigma_-$ can be obtained from $\sigma_+$ by adding the vector $i\in I_+$ which is missing in $\sigma_+$ and deleting some vector $k\in I_-$.
\end{definition}

It is proved in \cite[Section~7.1]{GKZbook} that one can obtain one triangulation of $\Sigma_{\pm}$ from another by replacing all essential cones of one triangulation with those of another. The associated toric Deligne--Mumford stacks $\PP_{\Sigma_{\pm}}$ are then related by an Atiyah flop that is a composition of a weighted blow-down and a weighted blow-up:
\begin{align*}
	\xymatrix{
	 & \PP_{\hat{\Sigma}}\ar[dr]^{f_+}\ar[dl]_{f_-} & \\
	 \PP_{\Sigma_-}\ar@{-->}[rr] & & \PP_{\Sigma_+}\rlap{.} \\
	}
\end{align*}
Here $\PP_{\hat{\Sigma}}$ is a common blow-up of $\PP_{\Sigma_{\pm}}$ defined as follows. The linear relation can be rewritten as 
\begin{align*}
	\sum_{i\in I_+}h_i v_i=-\sum_{i\in I_-}h_i v_i.
\end{align*}
We denote this vector by $\hat{v}$. We then define $\hat{\Sigma}$ to be the fan obtained by replacing all essential cones of $\Sigma_{\pm}$ by cones of the form $\mathcal{F}\cup\{\hat{v}\}\cup(I\backslash\{i_+,i_-\})$, where $i_{\pm}\in I_{\pm}$.

The behavior of twisted sectors under the wall-crossing was studied in \cite[Section~4]{MellinBarnes} and \cite[Section~6.2.3]{CIJ}.

\begin{definition}\label{adjacent twisted sectors}
        Let $\gamma_{\pm}\in\Box(\Sigma_{\pm})$ be two essential twisted sectors. We say that $\gamma_-$ is \textit{adjacent} to $\gamma_+$ if there exists a pair of essential maximal cones $\sigma_{\pm}$ in $\Sigma_{\pm}$ such that $\sigma(\gamma_+)$ and $\sigma(\gamma_-)$ are subcones of $\sigma_+$ and $\sigma_-$, respectively.
\end{definition}

\begin{lemma}\label{choice of lifting}
  Let $\Sigma_{\pm}$ be two adjacent triangulations. Then there exists a choice of the lifting $\Box(\Sigma_{\pm})\rightarrow\KK_c^{\pm}$ such that for any pair of adjacent twisted sectors $\gamma_+$ and $\gamma_-$, the liftings $\widetilde{\gamma_+}$ and $\widetilde{\gamma_-}$
% query widetilde
  differs by a rational multiple of the defining linear relation $h=(h_1,\ldots,h_n)$ of the circuit $I$ that corresponds to the wall-crossing $\Sigma_+\rightarrow\Sigma_-$.
\end{lemma}
\begin{proof}
        The proof is similar to that of \cite[Proposition 4.4(ii)]{MellinBarnes}. We begin with an arbitrary essential twisted sector $\gamma_+\in\Box(\Sigma_{+}^{\es})$ and an arbitrary lifting $\gamma_+=\sum_{j=1}^n (\gamma_+)_j v_j$. We write $\sigma(\gamma_+)\subseteq\mathcal{F}\sqcup I\backslash i$, where $\mathcal{F}$ is a separating set and $i\in I_+$. For any $k\in I_-$, we take a rational number $q\in\Q$ such that $(\gamma_+)_k + q h_k\in\Z$. Then $((\gamma_+)_j)$ and $((\gamma_+)_j+qh_j)$ differ by a rational multiple of $h$. We denote the associated twisted sector of the latter by $\gamma_-$. It is then clear that $\sigma(\gamma_-)\subseteq\mathcal{F}\sqcup I\backslash k$; hence $\gamma_-$ is an essential twisted sector of $\Sigma_-$. Moreover, since $\sigma(\gamma_{\pm})$ share the same separating set, $\sigma_+$ and $\sigma_-$ are adjacent. 
	
	Now we have proved that by adding an appropriate rational multiple of $h$ to a lifting of an essential twisted sector in $\Sigma_+$, we get a lifting of an essential twisted sector in $\Sigma_-$. It remains to show that any essential twisted sector in $\Sigma_-$ can be obtained in this way. To see this, note that the procedure above is invertible (\textit{i.e.}, adding $-q\cdot h$ to the lifting), so if we start with some $\gamma_+$, apply the procedure above from $\Sigma_+$ to $\Sigma_-$ and back, we recover the original twisted sector. By switching the roles of $\Sigma_+$ and $\Sigma_-$, we deduce that any essential twisted sector in $\Sigma_-$ can be obtained from this procedure.
\end{proof}

\begin{remark}\label{liftings differ by h}
        In the proof above, we see that for a fixed essential twisted sector $\gamma_+$ of $\Sigma_+$ and $k\in I_-$, the lifting we constructed for the adjacent twisted sector is not unique due to the freedom of the condition $(\gamma_+)_k + q h_k\in\Z$. However, it is important to note that any two such liftings differ by an integral multiple of~$h$. In fact, whenever we have two liftings $\gamma_+ +q_1 h$ and $\gamma_+ +q_2 h$, they both define the same twisted sector of $\Sigma_-$ if and only if the fractional part of the corresponding coordinates are equal. This means that $(q_1-q_2)h$ should have integral coordinates. The primitivity of $h$ then forces $q_1-q_2$ to be an integer.
\end{remark}

\section{Analytic continuation of Gamma series}\label{sec.ac}

In this section, we compute the analytic continuation of Gamma series solutions to $\bbGKZ(C,0)$.

Following \cite{BHan}, the Gamma series solution to $\bbGKZ(C,0)$ associated to a triangulation $\Sigma$ is given by
\begin{align*}
	\Gamma_c=\bigoplus_{\gamma\in\Box(\Sigma)}\sum_{l\in L_{c,\gamma}}\prod_{j=1}^n\frac{x_j^{l_j+\frac{D_j}{2\pi {\rm i}}}}{\Gamma\left(1+l_j+\frac{D_j}{2\pi {\rm i}}\right)}, 
\end{align*}
and there exists a point $\hat{\psi}$ in the maximal cone $C_{\Sigma}$ of the secondary fan corresponding to $\Sigma$ such that the series converges absolutely on the open set
\begin{align*}
	U_{\Sigma}=\left\{\left(x_j\right)\in\C^{n}:\ \left(-\log\left|x_j\right|\right)\in \hat{\psi}+C_{\Sigma},\ \arg(\mathbf{x})\in (-\pi,\pi)^n\right\}, 
\end{align*}
which should be thought of as a neighborhood of the large radius limit point corresponding to $\Sigma$.

We introduce some additional notation that will be used later. We define $L_{c,\gamma,\sigma}$ to be the subset of $L_{c,\gamma}$ with the additional property that  the  set 
\begin{align*}
	I(l):=\left\{i:\ l_i\in\Z_{<0}\right\}\sqcup\sigma(\gamma)=\left\{i:\ l_i\not\in\Z_{\geq 0}\right\}
\end{align*}
is a subcone of the maximal cone $\sigma$. Note that an element $l\in L_{c,\gamma}$ has a non-zero contribution to the series if and only if it lies in one of the subsets $L_{c,\gamma,\sigma}$. Along the same line as the proof of \cite[Proposition 3.8]{BHan}, we can prove the following result. See also \cite[Proposition 2.8]{BH} for a similar result for the usual GKZ systems.

\begin{proposition}\label{convergence}
        For each maximal cone $\sigma$, the subseries
	\begin{align*}
		\bigoplus_{\gamma\in\Box(\Sigma)}\sum_{l\in L_{c,\gamma,\sigma}}\prod_{j=1}^n\frac{x_j^{l_j+\frac{D_j}{2\pi {\rm i}}}}{\Gamma\left(1+l_j+\frac{D_j}{2\pi {\rm i}}\right)}
	\end{align*}
	is absolutely and uniformly convergent on compacts in the region
	\begin{align*}
		U_{\sigma}=\left\{\left(x_j\right)\in\C^{n}:\ \left(-\log\left|x_j\right|\right)\in \hat{\psi}_{\sigma}+C_{\sigma},\ \arg(\mathbf{x})\in (-\pi,\pi)^n\right\}, 
	\end{align*}
	where $C_{\sigma}$ denotes the union of all maximal cones $C_{\Sigma}$ in the secondary fan that corresponds to triangulations $\Sigma$ such that $\sigma\in\Sigma$, and $\hat{\psi}_{\sigma}$ is a point in $C_{\sigma}$.
\end{proposition}

\begin{remark}\label{notation1}
	More generally, for any subset of maximal cones $J$ of $\Sigma$, the subseries taken over the union of all $L_{c,\gamma,\sigma}$ for $\sigma\in J$ converges absolutely and uniformly on compacts in  $U_J:=\cap_{\sigma\in J}U_{\sigma}$. In particular, the open set $U_{\Sigma}$ is a subset of the intersection of the $U_{\sigma}$ for all $\sigma\in\Sigma$.
\end{remark}

Furthermore, we define $L_{c,\gamma}^{\es}$ to be the union of $L_{c,\gamma,\sigma}$ for all $\sigma\in\Sigma^{\es}$. Note that $L_{c,\gamma}^{\es}$ is non-empty only if $\gamma\in\Box(\Sigma^{\es})$. We define the \textit{essential part} $\Gamma_c^{\es}=\oplus_{\gamma}\Gamma_{c,\gamma}^{\es}$ of the Gamma series $\Gamma_c$ to be the subseries that consists of terms corresponding to $l\in L_{c,\gamma}^{\es}$, namely
\begin{align*}
	\Gamma_c^{\es}:=\bigoplus_{\gamma}\sum_{l\in L_{c,\gamma}^{\es}}\prod_{j=1}^n\frac{x_j^{l_j+\frac{D_j}{2\pi {\rm i}}}}{\Gamma\left(1+l_j+\frac{D_j}{2\pi {\rm i}}\right)}, 
\end{align*}
and the \textit{non-essential part} to be $\Gamma_c-\Gamma_c^{\es}$.

Henceforth, we will add superscripts $\pm$ to the notation defined above to distinguish between Gamma series associated to different triangulations $\Sigma_{\pm}$.

The main goal of this section is to compute the analytic continuation of the Gamma series solution $\Gamma^{+}$ to $\bbGKZ(C,0)$ defined on $U_{\Sigma_+}$ along the following path (see Figure~\ref{path}) to $U_{\Sigma_-}$:
\begin{itemize}
	\item The start and end points $x_{\pm}\in U_{\Sigma_{\pm}}$ should be chosen so that both of them lie in the open set $U_{\Sigma_+\cap\Sigma_-}$\footnote{Here $\Sigma_+\cap\Sigma_-$ denotes the set of common maximal cones of $\Sigma_{\pm}$; see Remark~\ref{notation1}.} and satisfy $\arg(x_+)_j=\arg(x_-)_j$ and $-\log|(x_+)_j|+\log|(x_-)_j|=Ah_j$ for any $j$ and some constant $A>0$.\footnote{This condition is equivalent to saying that the line connecting $\log|(x_+)_j|$ and $\log|(x_-)_j|$ is perpendicular to the wall that separates $\Sigma_+$ and $\Sigma_-$ in the secondary fan.}
	\item The path $x(u)$, $0\leq u\leq 1$,  from $x_+$ to $x_-$ is chosen so that for any $u\in[0,1]$, 
	\begin{align*}
		\arg(x(u)_j)=\arg(x_+)_j=\arg(x_-)_j,\\
		\log\left|x(u)_j\right|=(1-u)\log\left|(x_+)_j\right|+u\log\left|(x_-)_j\right|.
	\end{align*}
	Moreover, we require the argument of the  auxiliary variable
	\begin{align*}
		y:=e^{\ii\pi\sum_{j\in I_-}h_j}\prod_{j=1}^n x_j^{h_j}
	\end{align*}
	to be  restricted in the interval $(-2\pi,0)$ along the path. The existence of this path is guaranteed by the choice of $x_{\pm}$.
\end{itemize}

\begin{figure*}[!htbp]
	\centering
	\includegraphics[width=0.6\textwidth]{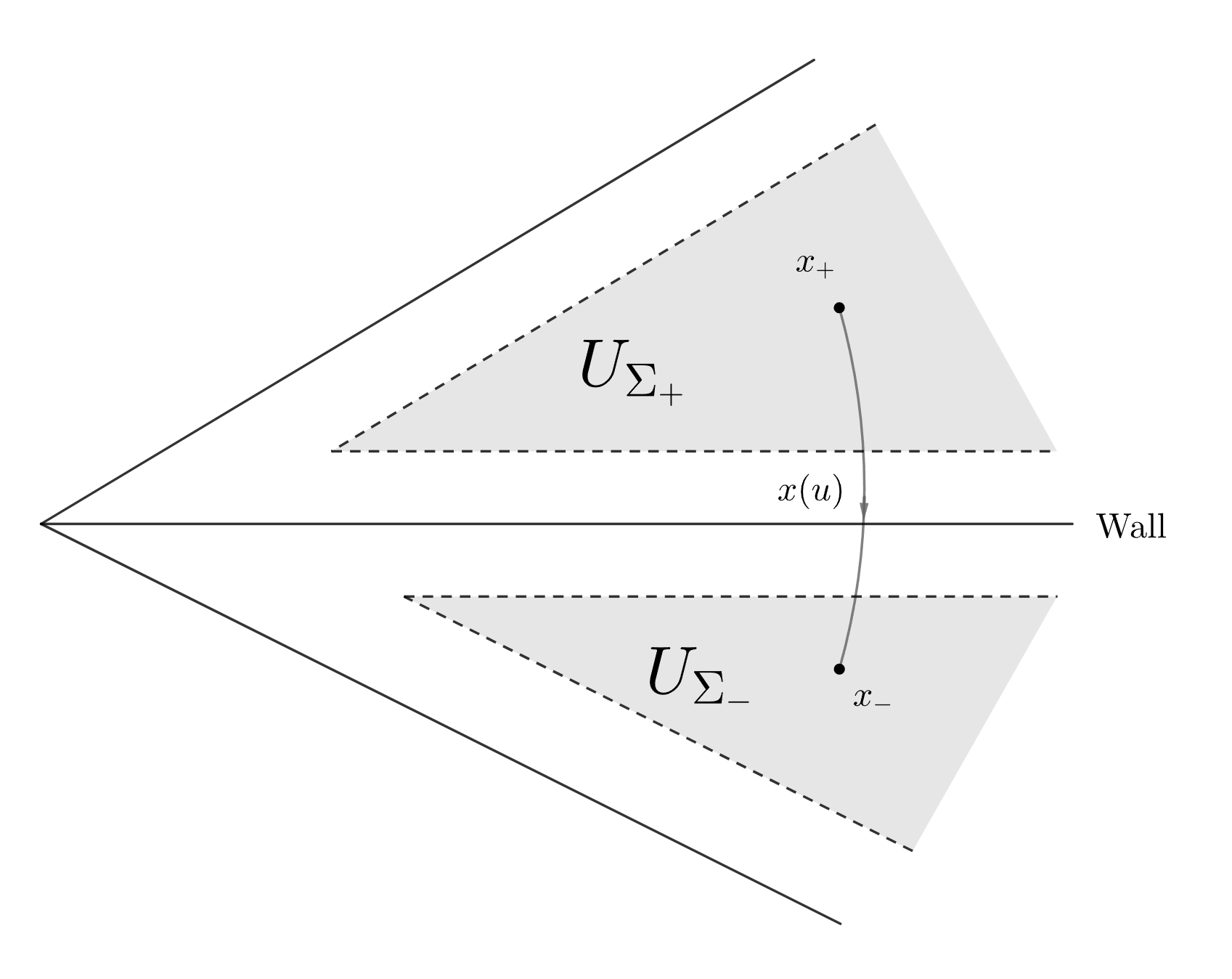}
	\caption{Path of analytic continuation}
	\label{path}
\end{figure*}

\begin{remark}
	The restriction on the argument of the variable $y$ is imposed to avoid introducing monodromy during the process of analytic continuation.
\end{remark}

 The main idea comes from \cite{MellinBarnes}, where the technique of Mellin--Barnes integrals is used to compute the analytic continuation for the usual GKZ systems. The main difference is that while they worked with $K$-theory-valued solutions, we work with the orbifold cohomology-valued solutions which makes the computation simpler, inspired by the approach of \cite{CIJ}.
 
\begin{remark}
	Let us make a remark here that throughout the remainder of this section, we think of the symbols $D_i$ as \textit{generic complex numbers}. The reason will be clear once we arrive at the proof of Theorem~\ref{conj for usual system}.
\end{remark}

In the following, we study the analytic continuation of the essential part and the non-essential part separately. The latter case is easier.

\begin{proposition}\label{non-eseential AC}
	The analytic continuation of\, $\Gamma_c^{+}-\Gamma_c^{+,\es}$ along the path $\lambda$ is  equal to $\Gamma_c^{-}-\Gamma_c^{-,\es}$.
\end{proposition}
\begin{proof}
	By definition, each single term $\prod_{j=1}^n\frac{x_j^{l_j+\frac{D_j}{2\pi {\rm i}}}}{\Gamma(1+l_j+\frac{D_j}{2\pi {\rm i}})}$ in the non-essential part $\Gamma_c^{\pm}-\Gamma_c^{\pm,\es}$ corresponds to some $l\in \bigcup_{\sigma\in\Sigma_+\cap\Sigma_-}L_{c,\gamma,\sigma}$. According to the choice of $x_{\pm}$, both of them lie in $U_{\Sigma_+\cap\Sigma_-}$; the convexity then ensures that the whole path $x(u)$ is contained in this open set. Therefore, by Proposition~\ref{convergence}, the non-essential parts $\Gamma_c^{\pm}-\Gamma_c^{\pm,\es}$ are analytic on a open set which contains the analytic continuation path. This finishes the proof.
\end{proof}

The rest of this section will be devoted to the continuation of the essential part $\Gamma_c^{+,\es}$. From now on, we fix a twisted sector $\gamma\in\Box(\Sigma_+)$ such that $L_{c,\gamma}$ is non-empty\footnote{If the set $L_{c,\gamma}$ is empty, then the corresponding component $\Gamma_{c,\gamma}$ is equal to zero, thus there is no need for analytic continuation.} and look at the corresponding component of the Gamma series.

Recall from Section 2 that $h$ is the primitive integral linear relation associated to the wall-crossing from $\Sigma_+$ to $\Sigma_-$. It is clear from the definition that $h$ acts on the lattice $L_{c,\gamma}^{+}$ by translation. It is also clear that if $l\in L_{c,\gamma}^{\es,+}$, then for any $m\geq 0$, the translation $l+mh$ also lies in $L_{c,\gamma}^{\es,+}$ (see \cite[Proposition 4.7]{MellinBarnes} for a similar result for usual GKZ systems). Hence the  subset of $L_{c,\gamma}^{\es,+}$
\begin{align*}
	\tilde{L}_{c,\gamma}^{\es,+}:=\left\{l\in L_{c,\gamma}^{\es,+}:\ l-h\not\in L_{c,\gamma}^{\es,+}\right\}
\end{align*}
is well defined, and $L_{c,\gamma}^{\es,+}=\tilde{L}_{c,\gamma}^{\es,+}+\Z_{\geq 0}h$. We can define $\tilde{L}_{c,\gamma^{\prime}}^{\es,-}$ in the same way with an appropriate change of signs.

With this notation, the essential part $\Gamma_{c,\gamma}^{+,\es}$ is equal to
\begin{align*}
	\sum_{l\in L_{c,\gamma}^{\es,+}}\prod_{j=1}^n\frac{x_j^{l_j+\frac{D_j}{2\pi {\rm i}}}}{\Gamma\left(1+l_j+\frac{D_j}{2\pi {\rm i}}\right)}=\sum_{l^{\prime}\in \tilde{L}_{c,\gamma}^{\es,+}}\sum_{m=0}^{\infty}\prod_{j=1}^n\frac{x_j^{l_j^{\prime}+mh_j+\frac{D_j}{2\pi {\rm i}}}}{\Gamma\left(1+l_j^{\prime}+mh_j+\frac{D_j}{2\pi {\rm i}}\right)}.
\end{align*}
By applying the Euler identity $\Gamma(z)\Gamma(1-z)=\frac{\pi}{\sin(\pi z)}$, we can rewrite the product as
\begin{align*}
	\prod_{j=1}^n x_j^{l_j^{\prime}+\frac{D_j}{2\pi {\rm i}}}\cdot \frac{\prod_{j\in I_-}\frac{\sin\left(\pi\left(-l_j^{\prime}-\frac{D_j}{2\pi {\rm i}}\right)\right)}{\pi}\Gamma\left(-l_j^{\prime}-mh_j-\frac{D_j}{2\pi {\rm i}}\right)}{\prod_{j\not\in I_-}\Gamma\left(1+l_j^{\prime}+mh_j+\frac{D_j}{2\pi {\rm i}}\right)}\cdot\left((-1)^{\sum_{j\in I_-}h_j}\prod_{i=1}^n x_j^{h_j}\right)^m.
\end{align*}

Now we consider
\begin{align*}
  I(s)=-\prod_{j=1}^n x_j^{l_j^{\prime}+\frac{D_j}{2\pi {\rm i}}}\frac{\prod_{j\in I_-}\frac{\sin\left(\pi\left(-l_j^{\prime}-\frac{D_j}{2\pi {\rm i}}\right)\right)}{\pi}\Gamma\left(-l_j^{\prime}-sh_j-\frac{D_j}{2\pi {\rm i}}\right)}{\prod_{j\not\in I_-}\Gamma\left(1+l_j^{\prime}+sh_j+\frac{D_j}{2\pi {\rm i}}\right)}\Gamma(-s)\Gamma(1+s)\left(e^{{\rm i}\pi}y\right)^s, 
\end{align*}
where $y=e^{\ii\pi\sum_{j\in I_-}h_j}\prod_{j=1}^n x_j^{h_j}$. There are two types of poles of $I(s)$:
\begin{itemize}
	\item integers $s=m\in\Z$, which come from the factor $\Gamma(-s)\Gamma(1+s)$, 
	\item $s=p_{k,w}:=-\frac{1}{h_k}\left(l_k^{\prime}+\frac{D_k}{2\pi{\rm i}}\right)+\frac{w}{h_k}$, $k\in I_-$, $w\in\Z_{\geq 0}$, which come from the factor $\prod_{j\in I_-}\Gamma(-l_j^{\prime}-sh_j-\frac{D_j}{2\pi {\rm i}})$.
\end{itemize}

 The reason why we use this specific form of $I(s)$ is that the residue of $I(s)$ at integers $s=m\in\Z$ is exactly
\begin{align*}
	\res_{s=m\in\Z}I(s)=\prod_{j=1}^n\frac{x_j^{l_j^{\prime}+mh_j+\frac{D_j}{2\pi {\rm i}}}}{\Gamma\left(1+l_j^{\prime}+mh_j+\frac{D_j}{2\pi {\rm i}}\right)}, 
\end{align*}
that is, what we have in the original Gamma series. Also note that the residues at non-positive integers are in fact zero, which follows directly from the definitions of $l^{\prime}$ and $\tilde{L}_{c,\gamma}^{\es,+}$.

The next step is to compute the residues of $I(s)$ at $p_{k,w}$. An application of the Euler identity together with an elementary computation show that
\begin{align*}
	I(s)&=-\frac{\pi e^{{\rm i}\pi s}}{\sin(-\pi s)}\prod_{j\in I_-}\frac{\sin\left(\pi\left(-l_j^{\prime}-\frac{D_j}{2\pi {\rm i}}\right)\right)}{\sin\left(\pi\left(-l_j^{\prime}-sh_j-\frac{D_j}{2\pi {\rm i}}\right)\right)}e^{{\rm i}\pi\left(\sum_{j\in I_-}h_j\right)s}\prod_{j=1}^n\frac{x_j^{l_j^{\prime}+sh_j+\frac{D_j}{2\pi{\rm i}}}}{\Gamma\left(1+l_j^{\prime}+sh_j+\frac{D_j}{2\pi {\rm i}}\right)}\\
	&=\frac{2\pi{\rm i}}{1-e^{-2{\rm i}\pi s}}\prod_{j\in I_-}\frac{1-e^{-2{\rm i}\pi \left(l_j^{\prime}+\frac{D_j}{2\pi{\rm i}}\right)}}{1-e^{-2{\rm i}\pi \left(l_j^{\prime}+sh_j+\frac{D_j}{2\pi{\rm i}}\right)}}\prod_{j=1}^n\frac{x_j^{l_j^{\prime}+sh_j+\frac{D_j}{2\pi{\rm i}}}}{\Gamma\left(1+l_j^{\prime}+sh_j+\frac{D_j}{2\pi {\rm i}}\right)}.
\end{align*}
It suffices to look at the factor $1/(1-e^{-2{\rm i}\pi(l_k^{\prime}+sh_k+\frac{D_k}{2\pi{\rm i}})})$, whose residue at $p_{k,w}$ is equal to $\frac{1}{2\pi{\rm i}h_k}$. Putting all these together, we get
\begin{align*}
	\res_{s=p_{k,w}}I(s)
	=&\frac{1-e^{-2{\rm i}\pi \left(l_k^{\prime}+\frac{D_k}{2\pi{\rm i}}\right)}}{h_k\left(1-e^{-2{\rm i}\pi p_{k,w}}\right)}\prod_{\substack{ j\in I_- \\ j\not=k }}\frac{1-e^{-2{\rm i}\pi \left(l_j^{\prime}+\frac{D_j}{2\pi{\rm i}}\right)}}{1-e^{-2{\rm i}\pi \left(l_j^{\prime}+p_{k,w}h_j+\frac{D_j}{2\pi{\rm i}}\right)}}
	\cdot\prod_{j=1}^n\frac{x_j^{l_j^{\prime}+p_{k,w}h_j+\frac{D_j}{2\pi{\rm i}}}}{\Gamma\left(1+l_j^{\prime}+p_{k,w}h_j+\frac{D_j}{2\pi {\rm i}}\right)}.
\end{align*}
%% need cdot? 
Now we introduce new notation. We write $w=w_0\cdot(-h_k)+r$, where $w_0,r\in\Z$ and $0\leq r<-h_k$, and define
\begin{align*}
	l^{\prime\prime}:=l^{\prime}+\frac{l_k^{\prime}-r}{-h_k}h.
\end{align*}
Then we have
\begin{align*}
	1+l_j^{\prime}+p_{k,w}h_j+\frac{D_j}{2\pi{\rm i}}&=1+\left(l_j^{\prime}-\frac{l_k^{\prime}}{h_k}h_j+w\frac{h_j}{h_k}\right)+\left(\frac{D_j-\frac{h_j}{h_k}D_k}{2\pi{\rm i}}\right)\\
	&=1+\left(l_{j}^{\prime\prime}-w_0h_j\right)+\left(\frac{D_j-\frac{h_j}{h_k}D_k}{2\pi{\rm i}}\right).
\end{align*}

We denote the associated twisted sector by $\gamma^{(k,r)}\in\Box(\Sigma_-^{\es})$. The key observation is $l^{\prime\prime}\in \tilde{L}_{c,\gamma^{(k,r)}}^{\es,-}$.

\begin{lemma}
	Given $l^{\prime}\in \tilde{L}_{c,\gamma}^{\es,+}$, $k\in I_-$ and $0\leq r<-h_k$, there is a uniquely determined essential twisted sector $\gamma^{(k,r)}\in\Box(\Sigma_-)^{\es}$ such that $l^{\prime\prime}\in \tilde{L}_{c,\gamma^{(k,r)}}^{\es,-}$.
\end{lemma}
\begin{proof}
	The twisted sector $\gamma^{(k,r)}$ is defined as the associated twisted sector of $l^{\prime\prime}$ in the sense of Lemma~\ref{associated twisted sector}; \textit{i.e.}, $\gamma^{(k,r)}:=\sum_{j=1}^n \{l_j^{\prime\prime}\}v_j$.
	
	First we show that $\gamma^{(k,r)}$ is an essential twisted sector in $\Sigma_-$. This is equivalent to showing that $\{j:\ l_j^{\prime\prime}\not\in\Z\}$ is a subcone of an essential cone in $\Sigma_-$. Since $l^{\prime}\in \tilde{L}_{c,\gamma}^{\es,+}$, we can write
	\begin{align*}
		I(l^{\prime})=\left\{j:\ l_j^{\prime}\not\in\Z_{\geq 0}\right\}\subseteq \mathcal{F}\sqcup I\backslash\{i\}\in\Sigma_+^{\es}
	\end{align*}
	for some separated set $\mathcal{F}$ and $i\in I_+$. Take $j$ such that $l_j^{\prime\prime}\not\in\Z$. If $j\not\in I$, then $h_j=0$ and $l_j^{\prime\prime}=l_j^{\prime}$, so $j\in I(l^{\prime})$, and hence $j\in\mathcal{F}$. On the other hand, if $j\in I$, then $j\not=k$ because by the definition of $l^{\prime\prime}$, we have $l_k^{\prime\prime}=r\in\Z_{\geq 0}$, so $k\not\in I(l^{\prime\prime})$. Now we conclude that $I(l^{\prime\prime})\subseteq\mathcal{F}\sqcup I\backslash\{k\}$, which is an essential cone in $\Sigma_-$ because $k\in I_-$. This also shows that $l^{\prime\prime}$ lies in $L_{c,\gamma^{(k,r)}}^{\es,-}$.
	
	Next we show that $l^{\prime\prime}\in \tilde{L}_{c,\gamma^{(k,r)}}^{\es,-}$. It suffices to show that $l^{\prime\prime}+h\not\in L_{c,\gamma^{(k,r)}}^{\es,-}$. This follows from the construction of $l^{\prime\prime}$. Note that $l^{\prime\prime}$ is chosen such that $l_k^{\prime\prime}=r\in\Z_{\geq 0}$ while $l_k^{\prime\prime}+h_k=r+h_k\in\Z_{< 0}$. This implies that $I(l^{\prime\prime})\subseteq\mathcal{F}\sqcup I\backslash\{k\}$ while $I(l^{\prime\prime}+h)\not\subseteq\mathcal{F}\sqcup I\backslash\{k\}$. Therefore, $l^{\prime\prime}+h\not\in L_{c,\gamma^{(k,r)}}^{\es,-}$ and the proof is completed.
\end{proof}

The first two terms in the residue can be written as
\begin{align*}
	\frac{1-e^{-2{\rm i}\pi \left(l_k^{\prime}+\frac{D_k}{2\pi{\rm i}}\right)}}{h_k\left(1-e^{-2{\rm i}\pi \left(\frac{l_k^{\prime}-l_k^{\prime\prime}}{-h_k}-\frac{D_k}{2\pi{\rm i}h_k}\right)}\right)}\prod_{\substack{ j\in I_- \\ j\not=k }}\frac{1-e^{-2{\rm i}\pi \left(l_j^{\prime}+\frac{D_j}{2\pi{\rm i}}\right)}}{1-e^{-2{\rm i}\pi \left(l_j^{\prime\prime}+\frac{D_j-\frac{h_j}{h_k}D_k}{2\pi{\rm i}}\right)}}.
\end{align*}
We claim that this factor only depends on the twisted sectors $\gamma$ and $\gamma^{(k,r)}$. To see this, recall our choice of the lifting $\Box(\Sigma_{\pm})\rightarrow\KK_c$ made in Remark~\ref{choice of lifting}. It is then clear that $(l^{\prime}-l^{\prime\prime})-(\gamma-\gamma^{(k,r)})$ is also a rational multiple of $h$. However, the left side has integral coordinates, which forces the right side to be an integer multiple of $h$ due to the primitivity. This implies that the $\supth{k}$ coordinate $(l^{\prime}_k-l^{\prime\prime}_k)-(\gamma_k-\gamma^{(k,r)}_k)$ is an integer multiple of $h_k$, which means that 
\begin{align*}
	1-e^{-2{\rm i}\pi \left(\frac{l_k^{\prime}-l_k^{\prime\prime}}{-h_k}-\frac{D_k}{2\pi{\rm i}h_k}\right)}=1-e^{-2{\rm i}\pi \left(\frac{\gamma_k-\gamma^{(k,r)}_k}{-h_k}-\frac{D_k}{2\pi{\rm i}h_k}\right)}.
\end{align*}
Together with the facts that $l^{\prime}_j\equiv\gamma_j$ and $l^{\prime\prime}_j\equiv\gamma^{(k,r)}_j$ modulo $\Z$, this implies that the original factor is  equal to
\begin{align*}
	C_{\gamma^{(k,r)}}:=\frac{1-e^{-2{\rm i}\pi \left(\gamma_k+\frac{D_k}{2\pi{\rm i}}\right)}}{h_k\left(1-e^{-2{\rm i}\pi \left(\frac{\gamma_k-\gamma^{(k,r)}_k}{-h_k}-\frac{D_k}{2\pi{\rm i}h_k}\right)}\right)}\prod_{\substack{ j\in I_- \\ j\not=k }}\frac{1-e^{-2{\rm i}\pi \left(\gamma_j+\frac{D_j}{2\pi{\rm i}}\right)}}{1-e^{-2{\rm i}\pi \left(\gamma^{(k,r)}_j+\frac{D_j-\frac{h_j}{h_k}D_k}{2\pi{\rm i}}\right)}}; 
\end{align*}
hence the residue is
\begin{align*}
	\res_{s=p_{k,w}}I(s)=C_{\gamma^{(k,r)}}\cdot\prod_{j=1}^n\frac{x_j^{\left(l_j^{\prime\prime}-w_0 h_j\right)+\left(\frac{D_j-\frac{h_j}{h_k}D_k}{2\pi{\rm i}}\right)}}{\Gamma\left(1+\left(l_j^{\prime\prime}-w_0h_j\right)+\frac{D_j-\frac{h_j}{h_k}D_k}{2\pi{\rm i}}\right)}.
\end{align*}

Finally, we use the techniques of Mellin--Barnes integrals to finish the computation. First of all, we fix an $l^{\prime}\in \tilde{L}_{c,\gamma}^{\es,+}$ and do the analytic continuation to the corresponding subseries 
\begin{align*}
	\sum_{m=0}^{\infty}\prod_{j=1}^n\frac{x_j^{l_j^{\prime}+mh_j+\frac{D_j}{2\pi {\rm i}}}}{\Gamma\left(1+l_j^{\prime}+mh_j+\frac{D_j}{2\pi {\rm i}}\right)}.
\end{align*}

We consider the contour integral
\begin{align*}
	\frac{1}{2\pi{\rm i}}\int_{C} I(s)ds=\frac{1}{2\pi{\rm i}}\int_{a-{\rm i}\infty}^{a+{\rm i}\infty} I(s)ds. 
\end{align*}
Here the contour $C$ is parallel to the imaginary axis, and the real part $a$ of $C$ is a negative number satisfying $\epsilon<|a|<1$ for some $\epsilon>0$ which avoids any pole of the integrand.

Now by \cite[Lemma A.6]{MellinBarnes} (see also the proof of \cite[Theorem 4.10]{MellinBarnes}), the sum of residues of $I(s)$ at the poles on the right side of the contour $C$, 
\begin{align}\label{eq1}
	\sum_{m=0}^{\infty}\prod_{j=1}^n\frac{x_j^{l_j^{\prime}+mh_j+\frac{D_j}{2\pi {\rm i}}}}{\Gamma\left(1+l_j^{\prime}+mh_j+\frac{D_j}{2\pi {\rm i}}\right)}+\sum_{p_{k,w}\text{ on the right side of }C}\res_{p_{k,w}}I(s), 
\end{align}
is analytically continued to
\begin{align}\label{eq2}
	-\sum_{p_{k,w}\text{ on the left side of }C}\res_{p_{k,w}}I(s), 
\end{align}
that is, the negative of the sum of residues of $I(s)$ at poles on the left side of $C$. Note that for a fixed $l^{\prime}\in \tilde{L}_{c,\gamma}^{\es,+}$, the real part of the poles $p_{k,w}$ is bounded above; therefore, the second sum in~\eqref{eq1} is finite. Therefore, we can add the negative of it to both~\eqref{eq1} and~\eqref{eq2}, and deduce that
\begin{align*}
	\sum_{m=0}^{\infty}\prod_{j=1}^n\frac{x_j^{l_j^{\prime}+mh_j+\frac{D_j}{2\pi {\rm i}}}}{\Gamma\left(1+l_j^{\prime}+mh_j+\frac{D_j}{2\pi {\rm i}}\right)}
\end{align*}
is analytically continued to
\begin{align*}
	-\sum_{k\in I_-}\sum_{0\leq r<-h_k}C_{\gamma^{(k,r)}}\sum_{w_0=0}^{\infty}\prod_{j=1}^n\frac{x_j^{\left(l_j^{\prime\prime}-w_0 h_j\right)+\left(\frac{D_j-\frac{h_j}{h_k}D_k}{2\pi{\rm i}}\right)}}{\Gamma\left(1+\left(l_j^{\prime\prime}-w_0h_j\right)+\frac{D_j-\frac{h_j}{h_k}D_k}{2\pi{\rm i}}\right)}. 
\end{align*}

To proceed with the analytic continuation, we need some analytic results and estimates. The corresponding results in the setting of the usual GKZ systems can be found in \cite{MellinBarnes}. Indeed, the results in this paper
% query 
could be proved word for word following the argument therein.

We denote the intersection of the cone $C_{\Sigma_+}$ and $C_{\Sigma_-}$ in the secondary fan (\textit{i.e.}, the wall defined by the linear relation $h$) by $\widetilde{C}$.

\begin{lemma}\label{estimate1}
	For any $k,A>0$, there exists a $\widetilde{c}$ in the interior of\, $\widetilde{C}$ such that for any $l^{\prime}\in \widetilde{L}_{c,\gamma}^{\es,+}$,  we have
	\begin{align*}
		\sum_{j=1}^n l_j^{\prime}u_j\geq k\Vert l^{\prime}\Vert
	\end{align*}
	for any $u \in \widetilde{C}+\widetilde{c}+a$ and any $a\in\R^n$ with $\Vert a \Vert \leq A$.
\end{lemma}
\begin{proof}
	The proof is nearly the same as that of \cite[Lemma 4.11]{MellinBarnes}; the only difference is that the $\widetilde{L}_{c,\gamma}^{\es,+}$ is a shift of the $\mathcal{S}^{\prime}$ therein. See also the proof of \cite[Proposition 3.8]{BHan}.
\end{proof}

\begin{lemma}\label{estimate2}
	There exist an $A>0$ and a $\widetilde{c}\in\widetilde{C}$ such that the set
	\begin{align*}
		V_A:=\bigcup_{a:\Vert a \Vert <A}\left(\widetilde{C}+\widetilde{c}+a\right)
	\end{align*}
	intersects with $U_{\Sigma_{\pm}}$, and such that the integral
	\begin{align*}
		\int_{a-{\rm i}\infty}^{a+{\rm i}\infty} \sum_{l^{\prime}\in \widetilde{L}_{c,\gamma}^{\es,+}}I_{l^{\prime}}(s)ds
	\end{align*}
        $($where we use the subscript $l^{\prime}$ to emphasize the dependence of the integrand on $l^{\prime})$ is absolutely convergent on the region $U$ defined by
	\begin{align*}
		U=\{(x_j)\in\C^{n}:\ (-\log|x_j|)\in V_A,\ \ -2\pi<\arg y<0,\ \arg(\mathbf{x})\in (-\pi,\pi)^n\}. 
	\end{align*}
\end{lemma}
\begin{proof}
	The proof is parallel to that of \cite[Lemma 4.12]{MellinBarnes}. The contour is defined by $s=a+{\rm i}t$ for $t\in\R$. By applying \cite[Lemma A.5]{MellinBarnes}, we see that the integrand $I_{l^{\prime}}(s)$ is controlled by 
	\begin{align*}
		|y|^a e^{-(\pi +\arg y) t }(|t|+1)^{R+n/2}e^{-\pi|t|}\sum_{l^{\prime}\in \widetilde{L}_{c,\gamma}^{\es,+}}(4 e k)^{\Vert l^{\prime}\Vert}e^{\sum l_j^{\prime}\log|x_j|}
	\end{align*}
	for some $R>0$ independent of $l^{\prime}$.  We now apply Lemma~\ref{estimate1}; we can choose $\widetilde{c}$ in the interior $\widetilde{C}$ such that on the set $V_A$, we have
	\begin{align*}
		(4 e k)^{\Vert l^{\prime}\Vert}e^{\sum l_j^{\prime}\log|x_j|}\leq e^{-\epsilon\Vert l^{\prime}\Vert}
	\end{align*}
	for some $\epsilon>0$ and any $l^{\prime}\in \widetilde{L}_{c,\gamma}^{\es,+}$ and $x\in U$. Since $\Vert l^{\prime}\Vert$ is of polynomial growth, the sum over all $l^{\prime}$ is still controlled by an exponential function with negative exponent. Hence the integral is absolutely convergent and therefore defines an analytic function over $U$.
\end{proof}

Now Lemma~\ref{estimate2}, together with the fact that the sum of the second term in~\eqref{eq1} over all $l^{\prime}\in \tilde{L}_{c,\gamma}^{\es,+}$ is absolutely convergent (because it is a subseries of a finite sum of the original gamma series), allows us to deduce the following desired analytic continuation of $\Gamma_{c,\gamma}^+$: 
\begin{align*}
	-\sum_{k\in I_-}\sum_{0\leq r<-h_k}C_{\gamma^{(k,r)}}\sum_{l^{\prime\prime}\in \tilde{L}_{c,\gamma^{(k,r)}}^{\es,-}}\sum_{w_0=0}^{\infty}\prod_{j=1}^n\frac{x_j^{\left(l_j^{\prime\prime}-w_0 h_j\right)+\left(\frac{D_j-\frac{h_j}{h_k}D_k}{2\pi{\rm i}}\right)}}{\Gamma\left(1+\left(l_j^{\prime\prime}-w_0h_j\right)+\frac{D_j-\frac{h_j}{h_k}D_k}{2\pi{\rm i}}\right)}, 
\end{align*}
that is, 
\begin{align*}
	-\sum_{k\in I_-}\sum_{0\leq r<-h_k}C_{\gamma^{(k,r)}}\Gamma^{-,\es}_{c,\gamma^{(k,r)}}|_{D_j\to D_j-\frac{h_j}{h_k}D_k}, 
\end{align*}
where the subscript of $\Gamma^{-,\es}_{c,\gamma^{(k,r)}}|_{D_j\to D_j-\frac{h_j}{h_k}D_k}$ denotes substitution of $D_j$ by $D_j-\frac{h_j}{h_k}D_k$ and we have used the fact that $l^{\prime\prime}\in \tilde{L}_{c,\gamma^{(k,r)}}^{\es,-}$. Therefore,  we have proved the following result.

\begin{proposition}\label{essential AC}
	The analytic continuation of\, $\Gamma_{c,\gamma}^{+,\es}$ is given by
	\begin{align*}
		-\sum_{k\in I_-}\sum_{0\leq r<-h_k}C_{\gamma^{(k,r)}}\Gamma^{-,\es}_{c,\gamma^{(k,r)}}|_{D_j\to D_j-\frac{h_j}{h_k}D_k}. 
	\end{align*}
\end{proposition}

\section{Fourier--Mukai transforms}\label{sec.fm}

In this section, we compute the Fourier--Mukai transform associated to the toric wall-crossing $\PP_{\Sigma_-}\dashrightarrow\PP_{\Sigma_+}$ and match it with the analytic continuation computed in Section~\ref{sec.ac}.

The Fourier--Mukai transform associated to the flop $\PP_{\Sigma_-}\dashrightarrow \PP_{\Sigma_+}$ was studied by Borisov and Horja in \cite{Ktheory} and \cite{BH}. More precisely, they computed the images of the $K$-theory classes $R_i$ of $\PP_{\Sigma_-}$ under the pullback and pushforward functors and obtained the  formulae for the $K$-theoretic Fourier--Mukai transforms given below. See \cite[Propositions 5.1 and~5.2]{MellinBarnes}.

Before we state their result, we introduce some notation. Let $\gamma_+=\sum_{j}(\gamma_+)_j v_j$ be an essential twisted sector of $\Sigma_+$. We denote by $\mathcal{I}(\gamma_+)$ the set of complex numbers $t$ such that $e^{2\pi{\rm i}(\gamma_+)_j}\cdot t^{h_j}=1$ for some $j\in I_-$. It is proved in \cite[Section 4]{MellinBarnes} that this set is finite, consists of roots of unity and is in 1-1 correspondence with the adjacent twisted sectors of $\gamma_+$.

\begin{proposition}\label{BHformula}\leavevmode
\begin{enumerate}
        \item For any analytic function $\varphi$ and $J$ which is not a subcone of any essential cone, the class
	\begin{align*}
		\prod_{j\in J}(1-R_j)\varphi(R)
	\end{align*}
	remains unchanged under the Fourier--Mukai transform.
        \item For any analytic function $\varphi$, the image of the $K$-theory class $\varphi(R)=\varphi(R_1,\ldots,R_n)$ under the Fourier--Mukai transform $\FM$ is given by 
	\begin{align*}
		\FM(\varphi(R)) =(\FM(\varphi)(\mathcal{R}))(1), 
	\end{align*}
	where the function $\FM(\varphi)$ is defined as 
	\begin{align*}
		\FM(\varphi)(r)=\varphi(r)-\sum_{t\in\mathcal{I}}\int_{C_t}T\left(r,\hat{t}\right)\varphi\left(r\hat{t}^h\right)d\hat{t}, 
	\end{align*}
	where $T(r,\hat{t})=\frac{1}{2\pi{\rm i}(\hat{t}-1)}\prod_{j\in I_-}\frac{1-r_j^{-1}}{1-r_j^{-1}\hat{t}^{-h_j}}$ and $\mathcal{I}$ is a set of roots of unity of a large enough order such that it contains the set $\mathcal{I}(\gamma_+)$ defined above for all essential twisted sectors $\gamma_+$ of\, $\Sigma_+$. The contours $C_t$ for $t\in\mathcal{I}$ are circles defined in a way such that they include all poles of the integrand in the interior, and $\mathcal{R}_i$ is the endomorphism on $K_0(\PP_{\Sigma_+})$ defined by multiplication by $R_i$.
\end{enumerate}
\end{proposition}

In this section, we use the formulae of Borisov--Horja to compute the Fourier--Mukai transform of the non-essential part $\Gamma_c^{-}-\Gamma^{-,\es}_{c}$ and the essential part $\Gamma^{-,\es}_{c}$ separately. A comparison of the computation in this section with the one in the last section hence yields the FM=AC result for $\bbGKZ(C,0)$.

The non-essential part is easier to deal with.

\begin{proposition}\label{non-essential FM}
	The Fourier--Mukai transform $\FM(\Gamma_c^{-}-\Gamma^{-,\es}_{c})$ is given by
	\begin{align*}
		\FM\left(\Gamma_c^{-}-\Gamma^{-,\es}_{c}\right)=\Gamma_c^{-}-\Gamma^{-,\es}_{c}. 
	\end{align*}
\end{proposition}
\begin{proof}
	By definition, for each single term $\prod_{j=1}^n\frac{x_j^{l_j+\frac{D_j}{2\pi {\rm i}}}}{\Gamma(1+l_j+\frac{D_j}{2\pi {\rm i}})}$ in the non-essential part $\Gamma_c^{-}-\Gamma_c^{-,\es}$, the set $I(l)$ (and therefore $\{i:\ l_i\in\Z_{<0}\}$) is not a subcone of any essential cone. Hence it contributes a factor of the form $\prod_{j\in J}D_j$, where $J$ is not a subcone of any essential cone. Note that $D_j$ can be written as the product of $1-e^{D_j}$ with an invertible element; the original product can  therefore be written as
	\begin{align*}
		\prod_{j\in J}\left(1-e^{D_j}\right)\widetilde{\varphi}(D), 
	\end{align*}
	where $\widetilde{\varphi}$ is an analytic function. Taking the direct sum over all twisted sectors $\gamma$, we see that under the Chern character $K_0(\PP_{\Sigma_-})\xrightarrow{\lowsim}\oplus_{\gamma}H_{\gamma}$, the non-essential part $\Gamma_c^{-}-\Gamma_c^{-,\es}$ is exactly of the form $\prod_{j\in J}(1-R_j)\varphi(R)$ for some analytic function $\varphi$. Now the statement follows from the first part of  Proposition~\ref{BHformula}.
\end{proof}

In order to compare the Fourier--Mukai transform of the essential part $\Gamma^{-,\es}_{c}$ with the analytic continuation computed in the last section, we first rewrite the formula in the second part of Proposition~\ref{BHformula} by computing the residue of the integrand explicitly. Recall that $K_0(\PP_{\Sigma})$ is a semi-local ring whose maximal ideals are in 1-1 correspondence with twisted sectors $\gamma\in\Box(\Sigma)$.

\begin{proposition}\label{FMresidue}
	For any analytic function $\varphi$,  we have
	\begin{align*}
		\FM(\varphi)(z)_{\gamma}=
	\begin{cases}
		\hphantom{{-}}\varphi(z)  &  \text{if }\gamma\not\in\Box(\Sigma_+)^{\es},\\
		-\sum_{k\in I_-}\sum_{0\leq r<-h_k}C_{\gamma^{(k,r)}}\varphi\left(z\left(p^{(k,r)}\right)^h\right)  & \text{if }\gamma\in\Box(\Sigma_+)^{\es}, 
	\end{cases}
	\end{align*}
	where
        \[p^{(k,r)}:=e^{-2{\rm i}\pi \left(\frac{\gamma_k-\gamma^{(k,r)}_k}{-h_k}-\frac{D_k}{2\pi{\rm i}h_k}\right)}\]
        for $k\in I_-$, $0\leq r <-h_k$.
\end{proposition}
\begin{proof}
	Fix a twisted sector $\gamma\in\Box(\Sigma_+)$. We localize at the point $r_j=e^{D_j+2\pi{\rm i}\gamma_j}$ corresponding to $\gamma$. If $\gamma\not\in\Box(\Sigma_+)^{\es}$, then by the definition of $\mathcal{I}$, the integration kernel $T(r,\hat{t})$ has no pole inside the contours $C_t$, and therefore, the second term in the formula is equal to zero.
	
	Now suppose $\gamma\in\Box(\Sigma_+)^{\es}$. The poles of $T(r,\hat{t})$ are 1 together with points  $\hat{t}$ such that there exists a $k\in I_-$ with $r_k\hat{t}^{h_k}=1$,  where $r_k=e^{D_k+2\pi{\rm i}\gamma_k}$. The set of poles is then
	\begin{align*}
		\left\{e^{-\frac{1}{h_k}\left(2\pi{\rm i}\gamma_k+D_k\right)}\left(e^{\frac{2\pi{\rm i}}{h_k}}\right)^r\ :\ k\in I_-,0\leq r<-h_k\right\}.
	\end{align*}
	An elementary calculation shows that this set is in fact the same as
	\begin{align*}
		\left\{p^{(k,r)}:=e^{-2{\rm i}\pi\left(\frac{\gamma_k-\gamma^{(k,r)}_k}{-h_k}-\frac{D_k}{2\pi{\rm i}h_k}\right)}\ :\ k\in I_-, 0\leq r<-h_k\right\}, 
	\end{align*}
	where $\gamma^{(k,r)}$ is defined as in Section~\ref{sec.ac}.\footnote{We note that the $\gamma_{k}^{(k,r)}$ denotes the $\supth{k}$ coordinate of the \textit{lifting we chose for }$\gamma^{(k,r)}$ in the sense of Lemma~\ref{choice of lifting}, not necessarily equal to $\{l_k^{\prime\prime}\}$, which is zero.} To demonstrate this, we observe that according to Remark~\ref{liftings differ by h}, any two liftings differ by an integer multiple of $h$. Thus, $p^{(k,r)}$ does not depend on the choice of lifting, and the set $\{\gamma_k^{(k,r)}\}$ is equal to $\{0,1,\ldots,-h_k-1\}$ modulo $-h_k$. To compute the residue of $T(r,\hat{t})$ at these poles, it suffices to consider the factor $\frac{1-r_k^{-1}}{1-r_k^{-1}\hat{t}^{-h_k}}$. The residue is then equal to
	\begin{align*}
		2\pi{\rm i}\res_{\hat{t}=p^{(r,k)}}T\left(r,\hat{t}\right)&=\frac{1}{p^{(r,k)}-1}\prod_{\substack{j\in I_- \\ j\not=k}}\frac{1-e^{-D_j-2\pi{\rm i}\gamma_j}}{1-e^{-D_j-2\pi{\rm i}\gamma_j}\left(p^{(r,k)}\right)^{-h_j}}\cdot \frac{p^{(r,k)}(1-r_k^{-1})}{h_k}\\
		&=\frac{e^{2\pi{\rm i}\left(-\gamma_k-\frac{D_k}{2\pi{\rm i}}\right)}-1}{h_k(e^{\frac{2{\rm i}\pi}{-h_k}\left(\gamma^{(r,k)}_k-\gamma_k-\frac{D_k}{2\pi{\rm i}}\right)}-1)}\prod_{\substack{ j\in I_- \\ j\not=k }}\frac{1-e^{-2{\rm i}\pi \left(\gamma_j+\frac{D_j}{2\pi{\rm i}}\right)}}{1-e^{-2{\rm i}\pi \left(\gamma_j^{(r,k)}+\frac{D_j-\frac{h_j}{h_k}D_k}{2\pi{\rm i}}\right)}}\\
		&=C_{\gamma^{(k,r)}}.
	\end{align*}
Putting all these together, we obtain the desired result.
\end{proof}

\begin{corollary}\label{essential FM}
	The Fourier--Mukai transform $\FM(\Gamma^{-,\es}_{c})$ is given by
	\begin{align*}
		\FM(\Gamma^{-,\es}_{c})_{\gamma}=-\sum_{k\in I_-}\sum_{0\leq r<-h_k}C_{\gamma^{(k,r)}}\Gamma^{-}_{c,\gamma^{(k,r)}}|_{D_j\to D_j-\frac{h_j}{h_k}D_k}
	\end{align*}
	for each twisted sector $\gamma\in\Box(\Sigma_+)^{\es}$.
	\end{corollary}
\begin{proof}
	Apply Proposition~\ref{FMresidue} to $\Gamma_c^{-,\es}$.
\end{proof}

\begin{theorem}\label{conj for usual system}
	The following diagram commutes:
	\begin{align*}
		\xymatrix{
K_0(\PP_{\Sigma_+})^{\vee}\ar[d]^{\FM^{\vee}}\ar[r]^-{-\circ\Gamma_+} & \Sol(\bbGKZ(C,U_+))\ar[d]^{\MB} \\
K_0(\PP_{\Sigma_-})^{\vee}\ar[r]^-{-\circ\Gamma_-} & \Sol(\bbGKZ(C,U_-))\rlap{.} \\
    }
	\end{align*}
\end{theorem}
\begin{proof}
	This is equivalent to proving that for any linear function $\varphi\colon K_0(\PP_{\Sigma_+})\rightarrow \CC$, there holds
\begin{align*}
	\varphi\circ\FM\circ \Gamma_- = \MB(\varphi\circ\Gamma_+),
\end{align*}
which follows directly from Propositions~\ref{non-eseential AC},~\ref{essential AC} and~\ref{non-essential FM} and Corollary~\ref{essential FM}. It is important to note that this also explains why, in Section~\ref{sec.ac}, we made the assumption that all $D_i$ are generic complex numbers.
\end{proof}

\section{Compactly supported derived categories and dual systems}\label{sec.dual}

In this section, we make use of the duality result in \cite{BHan} to prove the analogous result for the dual system $\bbGKZ(C^{\circ},0)$.

Recall that there is a similarly defined Gamma series solution $\Gamma^\circ$ with values in the compactly supported orbifold cohomology $H^{*,c}_{\orb}=\bigoplus_\gamma H_\gamma^c$ to the dual system $\bbGKZ(C^{\circ},0)$. We define
\begin{align*}
\Gamma_c^{\circ}(x_1,\ldots,x_n) =\bigoplus_{\gamma} \sum_{l\in L_{c,\gamma}} \prod_{i=1}^n 
\frac {x_i^{l_i+\frac {D_i}{2\pi \ii }}}{\Gamma\left(1+l_i+\frac {D_i}{2\pi \ii }\right)}\left(\prod_{i\in\sigma}D_i^{-1}\right)F_{\sigma}, 
\end{align*}
where $\sigma$ is the set of $i$ with $l_i\in\Z_{<0}$ and the $F_{\sigma}$ are generators of $H^{*,c}_{\orb}$ as a module over $H_{\orb}^*$. 

In \cite{BHconj}, a compactly supported version $K_0^c(\PP_{\Sigma})$ of the $K$-theory of $\PP_{\Sigma}$ is defined in a purely combinatorial manner: The group $K_0^c(\PP_{\Sigma})$ is generated as a $K_0(\PP_{\Sigma})$-module by the $G_I$ for $I\in\Sigma$ such that $\sigma_{I}^{\circ}\subseteq C^{\circ}$ with the following relations for all $i\not\in I$:
\begin{align*}
	\left(1-R_i^{-1}\right)G_I=\begin{cases}G_{I\cup\{i\}}&\text{if }I\cup\{i\}\in\Sigma, \\ 0&\text{otherwise}.\end{cases}
\end{align*}
It is also proved in \cite{BHconj} that $K_0^c(\PP_{\Sigma})$ is naturally isomorphic to the Grothendieck group $K_0^c(\PP_{\Sigma})$ of the full subcategory $D^{\b}(\PP_{\Sigma})^c$ of $D^{\b}(\PP_{\Sigma})$ consisting of complexes supported on the compact subset $\pi_{\Sigma}^{-1}(0)$, where $\pi_{\Sigma}\colon\PP_{\Sigma}\rightarrow\Spec\C[C^{\vee}\cap N^{\vee}]$ is the structure morphism of the toric stack. Furthermore, there is a compactly supported Chern character $\ch^c\colon K_0^c(\PP_{\Sigma})\xrightarrow{\lowsim}H^c$, from which one can regard the cohomology-valued Gamma series solution as a $K_0^c(\PP_{\Sigma})$-valued solution.

\begin{remark}
	Another version of derived category with compact support has appeared in the literature, defined as the full subcategory of the bounded derived category of coherent sheaves $D^{\b}(X)$ consisting of complexes supported on \textit{arbitrary} compact subsets of $X$. The Grothendieck group of this category is isomorphic to the direct limit of the usual $K$-groups of all compact subsets of $X$ (see \textit{e.g.} \cite[Appendix~2]{shoemaker}). It is not clear whether it agrees with our version of compactly supported $K$-theory.
\end{remark}

The main theorem of this section is the following analogous result for the dual system $\bbGKZ(C^{\circ},0)$.

\begin{theorem}\label{conj for dual}
	The following diagram commutes:
	\begin{align}
    \xymatrix{
K_0^c\left(\PP_{\Sigma_+}\right)^{\vee}\ar[d]^{(\FM^c)^{\vee}}\ar[r]^-{-\circ\Gamma^{\circ}_+} & \Sol(\bbGKZ(C^{\circ}),U_+)\ar[d]^{\MB^c} \\
K_0^c\left(\PP_{\Sigma_-}\right)^{\vee}\ar[r]^-{-\circ\Gamma^{\circ}_-} & \Sol(\bbGKZ(C^{\circ}),U_-)\rlap{.} \\
    }
\end{align}
\end{theorem}
 
The first goal is to prove that there is a well-defined compactly supported $K$-theoretic Fourier--Mukai transform $\FM^c\colon K_0^c(\PP_{\Sigma_-})\rightarrow K_0^c(\PP_{\Sigma_+})$.

\begin{lemma}\label{small lemma}
    Let $f\colon\PP_{\hat{\Sigma}}\rightarrow\PP_{\Sigma}$ be the weighted blow-up along the closed substack $\PP_{\Sigma/\sigma_I}$, where $\sigma_I$ is a cone in $\Sigma$ $($not necessarily an interior cone$)$. Then we have
    \begin{align*}
        f^{-1}(\pi_{\Sigma}^{-1}(0))\subseteq \pi_{\hat{\Sigma}}^{-1}(0)\quad\text{and}\quad f(\pi_{\hat{\Sigma}}^{-1}(0))\subseteq \pi_{\Sigma}^{-1}(0). 
    \end{align*}
\end{lemma}
\begin{proof}
    The structure morphisms are compatible with the blow-up
    \begin{align*}
    	\xymatrix{
    	\PP_{\hat{\Sigma}}\ar[r]^{f}\ar[dr]_{\pi_{\hat{\Sigma}}} &  \PP_{\Sigma}\ar[d]^{\pi_{\Sigma}} \\
    	 & \Spec\C[C^{\vee}\cap N^{\vee}]\rlap{,}\\
    	}
    \end{align*}
    which induces the following Cartesian diagram: 
    \begin{align*}
    	\xymatrix{
    	\PP_{\hat{\Sigma}}\backslash \pi_{\hat{\Sigma}}^{-1}(0)\ar@{^(->}[r]\ar[d] & \PP_{\hat{\Sigma}}\ar[d]^{f} \\
    	\PP_{\Sigma}\backslash \pi_{\Sigma}^{-1}(0)\ar@{^(->}[r] & \PP_{\Sigma}\rlap{.} \\
    	}
    \end{align*}
    The lemma follows directly from the commutativity of this diagram.
\end{proof}

\begin{theorem}
    The Fourier--Mukai transform $\FM\colon D^{\b}(\PP_{\Sigma_-})\rightarrow D^{\b}(\PP_{\Sigma_+})$ maps $D^{\b}(\PP_{\Sigma_-})^c$ to $D^{\b}(\PP_{\Sigma_+})^c$.
\end{theorem}
\begin{proof}
    The center $\PP_{\Sigma_-/I_+}$ of the blow-up $\PP_{\hat{\Sigma}}\rightarrow\PP_{\Sigma_-}$ can be viewed as the zero locus of a regular section of the vector bundle $E=\bigoplus_{j\in I_+}\mathscr{O}_{\PP_{\Sigma_-}}(D_j)$; therefore, the blow-up $f_-$ can be decomposed as
    \begin{align*}
        \xymatrix{
\PP_{\hat{\Sigma}} \ar@{^(->}[r]^-{i}\ar[dr]_{f_-}  &  \PP_{\PP_{\Sigma_-}}(E^{\vee})\ar[d]^{p} \\
  &  \PP_{\Sigma_-}\rlap{,}        
        }
    \end{align*}
    where $\PP_{\PP_{\Sigma_-}}(E^{\vee})$ is the projective bundle associated to $E^{\vee}$ and $p$ is the projection. This is well known for varieties and can be proved for stacks similarly.

    Let $\mathscr{F}$ be an arbitrary coherent sheaf on $\PP_{\Sigma_-}$ supported on $\pi_{\Sigma_-}^{-1}(0)$. We first prove that $L(f_-)^*(\mathscr{F})$ is supported on $\pi_{\hat{\Sigma}}^{-1}(0)$. Since $i$ is a closed immersion, it suffices to show that $i_* L(f_-)^*\mathscr{F}$ is supported on the image of $\pi_{\hat{\Sigma}}^{-1}(0)$ under $i$. Note that we have
    \begin{align*}
        i_* L(f_-)^*\mathscr{F}=i_* Li^* p^*
\mathscr{F}\cong p^*\mathscr{F}\otimes^{L} i_*\mathscr{O}_{\PP_{\hat{\Sigma}}}; 
    \end{align*}
    here we used the facts that $p$ is flat, $i_*$ is exact and $R(f_-)_*\mathscr{O}_{\PP_{\hat{\Sigma}}}=\mathscr{O}_{\PP_{\Sigma_-}}$. From this we have
    \begin{align*}
        \supp(i_* L(f_-)^*\mathscr{F})&\subseteq \supp(p^*\mathscr{F})\cap\supp\left(i_*\mathscr{O}_{\PP_{\hat{\Sigma}}}\right)\subseteq p^{-1}(\supp(\mathscr{F}))\cap i\left(\PP_{\hat{\Sigma}}\right)\\
        &\subseteq p^{-1}\left(\pi_{\Sigma_-}^{-1}(0)\right)\cap i\left(\PP_{\hat{\Sigma}}\right). 
    \end{align*}
    It suffices to show that $p^{-1}(\pi_{\Sigma_-}^{-1}(0))\cap i(\PP_{\hat{\Sigma}})\subseteq i(\pi_{\hat{\Sigma}}^{-1}(0))$, which is again equivalent to $(f_-)^{-1}(\pi_{\Sigma_-}^{-1}(0))\subseteq \pi_{\hat{\Sigma}}^{-1}(0)$. Applying Lemma~\ref{small lemma}, we get the desired result.

    Now consider an arbitrary complex $\mathscr{F}^{\bullet}$ in $D^{\b}(\PP_{\Sigma_-})^c$. We argue by induction on the length of $\mathscr{F}^{\bullet}$. Denote the lowest degree of non-zero cohomology of $\mathscr{F}^{\bullet}$ by $i_0$; then there exists a distinguished triangle
    \begin{align*}
        \mathscr{H}^{i_0}(\mathscr{F}^{\bullet})[-i_0]\longrightarrow \mathscr{F}^{\bullet} \longrightarrow \mathscr{G}^{\bullet}\longrightarrow\mathscr{H}^{i_0}(\mathscr{F}^{\bullet})[1-i_0], 
    \end{align*}
    where the $\supth{i}$ cohomology of $\mathscr{G}^{\bullet}$ is isomorphic to that of $\mathscr{F}^{\bullet}$ for all $i>i_0$ and the $\supth{i_0}$ cohomology is zero. Applying the derived pullback, we get a distinguished triangle in $D^{\b}(\PP_{\hat{\Sigma}})$
    \begin{align*}
        L(f_-)^*\mathscr{H}^{i_0}(\mathscr{F}^{\bullet})[-i_0]\longrightarrow L(f_-)^*\mathscr{F}^{\bullet} \longrightarrow L(f_-)^*\mathscr{G}^{\bullet}\longrightarrow L(f_-)^*\mathscr{H}^{i_0}(\mathscr{F}^{\bullet})[1-i_0].
    \end{align*}
    By the induction assumption, both $L(f_-)^*\mathscr{H}^{i_0}(\mathscr{F}^{\bullet})$ and $L(f_-)^*\mathscr{G}^{\bullet}$ are supported on $\pi_{\hat{\Sigma}}^{-1}(0)$. Taking stalks of this distinguished triangle, we get that $L(f_-)^*\mathscr{F}^{\bullet}$ is also supported on $\pi_{\hat{\Sigma}}^{-1}(0)$.

    Next we look at the pushforward $R(f_+)_*\colon D^{\b}(\PP_{\hat{\Sigma}})\rightarrow D^{\b}(\PP_{\Sigma_+})$. Take an arbitrary complex $\mathscr{K}^{\bullet}$ in $D^{\b}(\PP_{\hat{\Sigma}})^c$. Consider the following Cartesian diagram used in the proof of the lemma:
    \begin{align*}
    	\xymatrix{
    	\PP_{\hat{\Sigma}}\backslash \pi_{\hat{\Sigma}}^{-1}(0)\ar@{^(->}[r]\ar[d] & \PP_{\hat{\Sigma}}\ar[d]^{f_+} \\
    	\PP_{\Sigma_+}\backslash \pi_{\Sigma_+}^{-1}(0)\ar@{^(->}[r] & \PP_{\Sigma}\rlap{.} \\
    	}
    \end{align*}
    We apply the flat base change formula and see that the restriction of $R(f_+)_*\mathscr{K}^{\bullet}$ to $\PP_{\Sigma_+}\backslash \pi_{\Sigma_+}^{-1}(0)$ is zero; that is, $R(f_+)_*\mathscr{K}^{\bullet}$ is supported on $\pi_{\Sigma_+}^{-1}(0)$. So $R(f_+)_*$ maps $D^{\b}(\PP_{\hat{\Sigma}})^c$ into $D^{\b}(\PP_{\Sigma_+})^c$.
\end{proof}

Therefore, the Fourier--Mukai transform induces an isomorphism $\FM^c\colon K_0^c(\PP_{\Sigma_-})\rightarrow K_0^c(\PP_{\Sigma_+})$ of compactly supported K-theories. Now we are ready to prove the main theorem of this section.

\begin{proof}[Proof of Theorem~\ref{conj for dual}]
Since $\FM$ is an equivalence of categories, we have
\begin{equation}\label{FM preserves Euler pairing}
\begin{split}
	\chi_-([\mathscr{E}_1^{\bullet}],[\mathscr{E}_2^{\bullet}])&=\sum_{i}(-1)^i \dim\Hom_{D^{\b}(\PP_{\Sigma_-})}(\mathscr{E}_1^{\bullet},\mathscr{E}_2^{\bullet}[i])\\
    &=\sum_{i}(-1)^i \dim\Hom_{D^{\b}(\PP_{\Sigma_+})}(\FM(\mathscr{E}_1^{\bullet}),\FM^c(\mathscr{E}_2^{\bullet})[i])\\
    &=\chi_+(\FM([\mathscr{E}_1^{\bullet}]),\FM^c([\mathscr{E}_2^{\bullet}])), 
\end{split}
\end{equation}
where $\chi_{\pm}$ denotes the Euler characteristic pairing on $\PP_{\Sigma_{\pm}}$. Hence $\FM$ preserves the Euler characteristic pairing.

To proceed, we need the following duality result for the pair of better-behaved GKZ systems, proved in \cite[Theorems 2.4 and~4.2]{BHan}. More precisely, there is a non-degenerate pairing between the solution spaces of $\bbGKZ(C,0)$ and $\bbGKZ(C^{\circ},0)$
	\begin{align*}
		\langle-,-\rangle\colon \Sol(\bbGKZ(C,0))\times \Sol(\bbGKZ(C^{\circ},0))\longrightarrow \C 
	\end{align*}
	that corresponds to the inverse of the Euler characteristic pairing in the large radius limit under the isomorphisms given by the Gamma series solutions.

Now consider the diagram 
\begin{align*}
    \xymatrix{
K_0^c\left(\PP_{\Sigma_+}\right)^{\vee}\ar[d]^{(\FM^c)^{\vee}}\ar[r]^-{-\circ\Gamma^{\circ}_+} & \Sol(\bbGKZ(C^{\circ}),U_+)\ar[d]^{\MB^c} \\
K_0^c\left(\PP_{\Sigma_-}\right)^{\vee}\ar[r]^-{-\circ\Gamma^{\circ}_-} & \Sol(\bbGKZ(C^{\circ}),U_-)\rlap{.} \\
    }
\end{align*}
To finish the proof, it suffices to show $(\FM^c)^{\vee}(g)\circ\Gamma^{\circ}_-=\MB^c(g\circ\Gamma^{\circ}_+)$ for any $g\in K_0^c(\PP_{\Sigma_+})^{\vee}$. Take an arbitrary $f\in K_0(\PP_{\Sigma_+})^{\vee}$. We have
\begin{align*}
    \langle(\FM)^{\vee}(f)\circ\Gamma_-, (\FM^c)^{\vee}(g)\circ\Gamma^{\circ}_-\rangle & = \chi_-^{\vee}((\FM)^{\vee}(f),(\FM^c)^{\vee}(g))\\
    & = \chi_+^{\vee}(f,g) \\
    & = \langle f\circ\Gamma_+, g\circ\Gamma^{\circ}_+ \rangle \\
    & = \langle \MB(f\circ\Gamma_+), \MB^c(g\circ\Gamma^{\circ}_+) \rangle \\
    & = \langle (\FM)^{\vee}(f)\circ\Gamma_-, \MB^c(g\circ\Gamma^{\circ}_+) \rangle.
\end{align*}
Here the first and third equalities follows from the duality of bbGKZ systems, the second equality follows from~\eqref{FM preserves Euler pairing}, the fourth equality follows from the definition of analytic continuation, and the last equality is Theorem~\ref{conj for usual system}.

Since the pairing of solutions is non-degenerate and the $(\FM)^{\vee}(f)\circ\Gamma_-$ span the whole solution space, we obtain the desired result.
\end{proof}

%%%%%%%%%%%%%%%%%%%%%
% References
%%%%%%%%%%%%%%%%%%%%%

\end{document}